\documentclass[11pt]{amsart}
\usepackage{amsbsy,amsfonts,amsmath,amssymb,amscd,amsthm,mathrsfs,calc,eucal}
\usepackage{graphicx,epsfig,color}
\usepackage{tikz,float} 
\usepackage[a4paper,text={6.1in,8in},centering,includefoot,foot=1.5in]{geometry}

\small\normalsize
\theoremstyle{plain}
\newtheorem{mainthm}{Theorem}

\newtheorem{theorem}{Theorem}[section]

\newtheorem{proposition}[theorem]{Proposition}

\newtheorem{definitions}[theorem]{Definitions}

\theoremstyle{definition}

\newtheorem{example}[theorem]{Example}

\numberwithin{equation}{section}
\let\oldmarginpar\marginpar
\renewcommand\marginpar[1]{\-\oldmarginpar[\raggedleft\footnotesize \textcolor{red}{#1}]
{\raggedright\footnotesize\textcolor{red}{#1}}}

\DeclareMathOperator{\IFS}{IFS} 
 
\DeclareMathOperator{\diam}{diam}
\DeclareMathOperator{\Ls}{Ls}
\begin{document}
\title[A bridge from topological properties into the long-term behavior]{Attractors as a bridge from topological properties to long-term behavior in dynamical systems}
\author[A. Sarizadeh]{Aliasghar Sarizadeh
\address{{Department of Mathematics,  Ilam University}}
\address{{Ilam, Iran.}}
\email{ali.sarizadeh@gmail.com}
\email{a.sarizadeh@ilam.ac.ir}}

\begin{abstract}
This paper refined and introduced some notations (namely attractors, 
physical attractors, proper attractors, topologically exact and topologically mixing) 
within  the context of relations.
We establish necessary and sufficient conditions, 
including that the phase space of a topologically exact system is an attractor for its inverse, 
and vice versa, and that a system is topologically mixing if and only if its phase space is a physical attractor. 

Through iterated function systems (IFSs), we illustrate classes of 
non-trivial topologically mixing and topologically exact IFSs. 
Additionally, we use IFSs to provide an example of topologically mixing system, generated 
by finite of homeomorphisms on a compact metric space, that is not topologically exact. 
These findings connect topological properties with attractor types, providing deeper insights into the long-term dynamics of such systems.
\end{abstract}

\maketitle
\setcounter{tocdepth}{1}

\section{Introduction: preliminaries and results}\label{chap1}

Attractors are central concept to understanding the long-term behavior of dynamical systems,  \cite{M,RT,R1,R2}. In an iterated function system (IFS), an attractor is a minimal compact set that attracts all compact sets in the Hausdorff metric as iterations approach infinite, \cite{Blr,Blw,Bv11,Bv13,bgms}.  However, this standard definition may overlook certain accumulation points, limiting its ability to capture complex dynamical behavior. To address this limitation, we introduce a refined notion of attractors, termed \emph{proper attractors}, which excludes such accumulation points. This paper investigates the interplay between proper attractors, topological exactness (also known as locally eventually onto (leo)), and physical attractors, topological mixing within the framework of relations, providing an insights into the asymptotic dynamics of dynamical systems. 

The study of relations offers a powerful framework for analyzing dynamical systems, particularly for non-invertible maps, whose dynamics are often intricately tied to the behavior of their inverses. By adopting the language of relations, we generalize traditional concepts of attractors and some topological properties, enabling a systematic exploration of their  dependencies. This approach facilitates a deeper understanding of how topological properties relate to the long-term behavior of dynamical systems.

To formalize this framework, consider two non-empty sets \(X\) and \(Y\). A \emph{relation} \(\phi: X \to Y\) is a subset of \(X \times Y\), where \(\phi(x) = \{ y : (x, y) \in \phi \}\) for \(x \in X\). 
Following \cite{A}, the image of a set \(A \subset X\) under \(\phi\) 
is given by \(\phi(A) = \bigcup_{a \in A} \phi(a)\), and the 
inverse relation is defined as \(\phi^{-1} = \{ (y, x) : (x, y) \in \phi \}\). 
A relation \(\phi\) is a map if \(\phi(x)\) is a singleton 
for each \(x \in X\) and \(\phi^{-1}(Y) = X\). A relation$\phi$ is \emph{surjective} if \(\phi(X) = Y\). For relations \(\phi: X \to Y\) and \(\psi: Y \to Z\), the composition \(\psi \circ \phi: X \to Z\) is the projection onto \(X \times Z\) of \((\phi \times Z) \cap (X \times \psi) \subset X \times Y \times Z\). For a relation \(\phi: X \to X\), we define \(\phi^{n+1} := \phi \circ \phi^n\) inductively, with \(\phi^0 := \text{id}_X\), the identity map.

Assuming \(X\) and \(Y\) are compact metric spaces, a relation \(\phi: X \to Y\) is \emph{closed} if it is a closed subset of \(X \times Y\), or equivalently, if \(\phi^{-1}(B)\) is closed for every closed set \(B \subseteq Y\) (i.e., \(\phi\) is upper semi-continuous). It is \emph{lower semi-continuous} if \(\phi^{-1}(B)\) is open for every open set \(B \subseteq Y\), and \emph{continuous} if it is both upper and lower semi-continuous. For a dynamical system \((X, \phi)\), where \(\phi: X \to X\) is a continuous relation on a compact metric space \((X, d)\), we denote by \(\mathcal{K}(X)\) the hyperspace of non-empty compact subsets of \(X\) equipped with the Hausdorff metric \(d_H\).

To capture accumulation points under iterations, we introduce a modified inferior 
limit for a sequence \(\{ B_n \}_{n \in \mathbb{N}}\) of non-empty compact sets \(B_n \subset X\). 
A non-empty proper compact set \(K\) belongs to \(\Ls^* B_n\) if, 
for every pair of disjoint open neighborhoods \(U \supset K\) and \(V\), the set
\[
N(B_n: U, K, V) = \{ i \in \mathbb{N} : U \cap \phi^i(B_j) \neq \emptyset \text{ and } V \cap \phi^i(B_j) = \emptyset \text{ for some } j \in \mathbb{N} \}
\]
is infinite. This modified limit focuses on persistent intersections and disjuncts simultaneously, 
offering a  characterization of long-term behavior distinct from the standard
Kuratowski lower limit.

The following definitions generalize and refine the notion of attractors for relations, introducing physical and proper attractors.
\begin{definitions}\label{defexact}
Let $\phi:X\to X $ be a continuous  relation. The metric space $X$ is:
\begin{itemize}
\item[(i)]
An \emph{attractor} for $\phi$
if, for every non-empty compact set $K$, 
$$d_H(\phi^i(K),X)\to 0 \  \text{as}\ i\to \infty.$$ 
\item[(ii)]
A \emph{physical attractor} for $\phi$
 if,   for every non-empty open set $U\subset X$, there exists a compact set $K\subset U$ so that
$$d_H(\phi^i(K),X)\to 0 \ \text{as} \ i\to \infty.$$
\item[(iii)]\label{defatt}
A \emph{proper attractor} for $\phi$
 if  $X$ is an attractor for $\phi$ and, for every family $\{B_n\}_{n\in\mathbb{N}}$ of non-empty subsets of $X$,
$\Ls^* B_n=\emptyset$. 
\end{itemize}
\end{definitions}
  Let \(\phi\) be a continuous relation on a compact metric space \((X, d)\), and let \(x_1, x_2 \in X\). Equip \(X\) with the metric
\begin{equation}\label{defmetric}
d_\phi(x_1, x_2) := \sup_{n \in \mathbb{N}} d_H(\phi^n(x_1), \phi^n(x_2)).
\end{equation}
A point \(x \in X\) is a \emph{sensitive point} of \(\phi\) if the identity map \(\text{id}_X: (X, d) \to (X, d_\phi)\) is discontinuous at \(x\). The relation \(\phi\) is \emph{sensitive} if there exists \(\epsilon > 0\) such that every non-empty open subset has \(d_\phi\)-diameter at least \(\epsilon\). A point \(x \in X\) is \emph{equicontinuous} if \(\text{id}_X: (X, d) \to (X, d_\phi)\) is continuous at \(x\). We denote the set of all equicontinuous points of \(\phi\) by \(\mathrm{Eq}(\phi)\). The relation \(\phi\) is \emph{equicontinuous} if \(\mathrm{Eq}(\phi) = X\).

Also, we refine the definitions of topological exactness and topological mixing for a continuous relation
 $\phi:X\to X $ on a metric space $X$. 
The relation $\phi$  is called \emph{topologically exact} 
 if for every open subset $W$ of $X$,
$
 \phi^j(W)=X,\ \text{for some} \ j\in \mathbb{N}.
$
The relation $ \phi$ is called 
\emph{topologically mixing} if for every open subsets $W$ and $V$, there exist $J\in \mathbb{N}$ so that
$\phi^j(W)\bigcap V\neq \emptyset$, for every  $ j \geq J$.

Under the compactness assumption of phase space, we establish the relationship between the aforementioned concepts.
\begin{mainthm}\label{topexactattractor}
Let $\phi$ be a continuous  relation on a compact metric space $(X,d)$.  Then:
\begin{itemize}
\item[(i)]\label{ite1t}
$X$ is a physical attractor   if and only if   $\phi$ is topologically mixing.
\item [(ii)] If  $X$ is an attractor for $\phi$, then it is a physical attractor.
\item[(iii)] If  $X$ is an attractor for $\phi$, then it is proper attractor.
\item[(iv)] If $X$ is an  attractor for $\phi$, 
then $\phi$  is equicontinuous i.e. $\mathrm{Eq}(\phi)=X$.
\item[(v)] If $X$ is an attractor for  $\phi$, then $\phi^{-1}$ is topologically exact, and vice versa.
\item[(vi)] If $X$ is an attractor for $\phi$, then  for every $\epsilon>0$ there exists  $n_\dagger\in\mathbb{N}$ 
so that $d_H(\phi^i(x),X)<\epsilon$, for every $i>n_\dagger$ and every $x\in X$. 
\end{itemize}
\end{mainthm}
In \cite{BC23}, under monotonic assumption of a continuous map $F:\mathcal{K}(X)\to \mathcal{K}(X)$ 
and $X$ being an attractor, the authors  prove  that for every $\epsilon>0$ there exists  $n_\dagger\in\mathbb{N}$ 
so that $d_H(F^i(x),X)<\epsilon$, for every $i>n_\dagger$ and every $x\in X$. 
Theorem~\ref{topexactattractor} present a different proof for Proposition 2.9 from \cite{BC23} by using the property of proper  attractors for relations.

To better understand the dynamics of the inverse of a topologically
 exact map, we remark the inverse limit set of a relation 
$\phi$ on a space $X$, as
$$
\mathcal{S}=\{(\dots,x_{-1},x_0)\in X^\mathbb{N}: \ x_{-i}\in \phi(x_{-i-1})\}.
$$
For sequence $(\dots,x_{-1},x_0) \in \mathcal{S}$, the point $x_0$ is known as the starting point.
\begin{mainthm}\label{contersen}
Let $\phi$ be a topologically exact map on a compact metric space $X$.  Then:
\begin{itemize}
\item[(i)]
For every $x_0$ in $X$, there exists a sequence
  $\textbf{\textsf{x}} \in \mathcal{S}$ 
with starting point $x_0$ so that the image of sequence  $\textbf{\textsf{x}} $ is dense in $X$.
\item[(ii)] For every $x_0,~y_0\in X$, there exist 
sequences $\textbf{\textsf{x}},~\textbf{\textsf{y}}\in\mathcal{S}$ so that $\liminf_id(x_{-i},y_{-i})=0.$
\end{itemize}
\end{mainthm}
Expanding maps serve as prominent examples of topologically exact maps.
Let $M$ be a smooth compact manifold and $f:M\to M$ be a map of class $C^1$. The map $f$ is expanding if there exists 
$\sigma>1$ and some Riemannian metric on $M$ so that 
$$
\|Df(x)v\|\geq \sigma\|v\|;\ \ \ \forall~x\in M\ \text{and}\ \forall~v\in T_xM.
$$
A well-known result in dynamical systems asserts that every expanding map $f$ is topologically exact map.
Consequently, $M$ is an attractor for $f^{-1}$.

We apply the above framework to iterated function systems (IFSs). 
Let \( \mathcal{F} = \{ f_1, \dots, f_k \} \) be an IFS of continuous maps on a compact metric space \( (X, d) \), with semigroup \( \mathcal{F}^+ \). 
We say that \( \mathcal{F} \)  is an IFS with repelling fixed point, if there exists  $f : X \to X$ belonging to $\mathcal{F}$ with a repelling fixed point $p$.

For \( \omega = (\omega_1, \omega_2, \dots) \in \{ 1, \dots, k \}^{\mathbb{N}} \), define \( f^0_\omega = \text{id}_X \), and
\(
f^n_\omega(x) = f_{\omega_n} \circ \dots \circ f_{\omega_1}(x).
\)
Set 
$$
\mathcal{F}^m=\{f_{\omega_m}\circ\dots\circ f_{\omega_1}:\ f_{\omega_i}\in \mathcal{F}, \text{foll all}\ i=1,\dots,m\},
$$
where $m\in \mathbb{N}$. Define the Hutchinson operator 
\( F: \mathcal{K}(X) \to \mathcal{K}(X) \) by \( F(A) = \bigcup_{i=1}^k f_i(A) \), 
and its extension \( \widetilde{F}: 2^X \to 2^X \) by \( \widetilde{F}(A) = \bigcup_{i=1}^k f_i(A) \).

The $\IFS\mathcal{F}$ is \emph{Topologically mixing} if, for every open \( W, V \subset X \), 
there exists \( J \in N \) such that \( \widetilde{F}^j(W) \cap V \neq \emptyset \) for all \( j \geq J \).
Also, the $\IFS\mathcal{F}$  is \emph{Topologically exact} if, for every open \( U \subset X \), 
there exists \( n \in N \) such that \( \widetilde{F}^n(U) = X \).
\begin{mainthm}\label{attfixed}
Let  $\mathcal{F}=\{f_1,\dots,f_k\}$ be a chain transitive IFS with the
shadowing property on a compact metric space $(X,d)$. 
Then  the phase space $X$ is a physical attractor for the Hutchinson operator of $\mathcal{F}$. 
\end{mainthm}
\begin{mainthm}\label{replling}
Let  $\mathcal{F}=\{f_1,\dots,f_k\}$ be a family of homeomorphisms on compact metric space $X$.
If $\mathcal{F}$ is backward minimal with repelling fixed point then it is topologically exact.
\end{mainthm}
The presence of a repelling fixed point in Theorem~\ref{replling} plays a critical role in promoting topological mixing to topological exactness, as it ensures that open sets expand to cover the entire space under iterations.

This paper is organized as follows. In Section~\ref{top}, we prove Theorem~\ref{topexactattractor} and Theorem~\ref{contersen}, establishing the relationships between topological exactness, mixing, and attractor types.
 In Section~\ref{sec3}, we apply our framework to IFSs, proving theorems that characterize topologically mixing and exact systems (e.g., Theorems~\ref{attfixed} and \ref{replling}) and providing a concrete example of a topologically mixing but not topologically exact IFS. 


%
%
%
%
%
%
%
%
%
%
%
%
%
%
%
%
\section{Topological exactness and attractors: proof of Theorem \ref{topexactattractor}}\label{top}
Next  proposition  characterizes attractors in terms of intersection of the iterations of compact sets with open sets.
It is remarkable that the IFS version of the following proposition, 
can be found in reference \cite{BC23}, but for 
 completeness of the comparison we state and reproof relation form of it.
\begin{proposition}\label{equform}
Let  $\phi$ be a continuous relation  on a compact metric space $X$. Then
$X$ is an attractor for $\phi$ if and only 
 if  for every $K\in\mathcal{K}(X)$   and  every open subset $W$  of $X$
 there exists $J_K\in \mathbb{N}$ so that
$\phi^j(K)\bigcap W\neq \emptyset$, for every  $ j \geq J_K$.
\end{proposition}
\begin{proof}
 Assume \(X\) is an attractor for \(\phi\). 
 For any open set \(W \subset X\) and point \(p \in W\), 
 there exists \(\epsilon > 0\) such that the ball \(B(p, \epsilon) \subseteq W\). 
 Since \(X\) is an attractor, for $K\in\mathcal{K}(X)$, \(d_H(\phi^j(K), X) \to 0\) as \(j \to \infty\). Thus, there exists \(J_K \in \mathbb{N}\) such that \(d_H(\phi^j(K), X) < \epsilon\) for all \(j > J_K\). If \(\phi^j(K) \cap W = \emptyset\) for some \(j > J_K\), then \(d_H(\phi^j(K), X) \geq d(p, \phi^j(K)) \geq \epsilon\), a contradiction. Hence, \(\phi^j(K) \cap W \neq \emptyset\) for all \(j \geq J_K\).

Conversely, let \(\epsilon > 0\),  \(K \in \mathcal{K}(X)\), and  \(\{B_i\}_{i \in \mathbb{N}}\) be an open cover of \(X\) with \(\diam(B_i) < \epsilon/2\). By compactness, there exists a finite subcover \(\bigcup_{j=1}^m B_{i_j} = X\). By assumption, for each \(i_j\), there exists \(n_{i_j} \in \mathbb{N}\) such that \(\phi^n(K) \cap B_{i_j} \neq \emptyset\) for all \(n > n_{i_j}\). Set \(J_K = \max\{n_{i_j} : j = 1, \ldots, m\}\). For \(n > J_K\), \(\phi^n(K) \cap B_{i_j} \neq \emptyset\) for all \(j = 1, \ldots, m\), so \(\phi^n(K)\) intersects every set in the cover, implying \(d_H(\phi^n(K), X) < \epsilon\). Thus, \(d_H(\phi^n(K), X) \to 0\) as \(n \to \infty\), and \(X\) is an attractor.
\end{proof}


\begin{proof}[Proof of Theorem~\ref{topexactattractor} (i)]
assume  $X$ is a physical attractor. Let $V,~W$ be non-empty open subsets.
Since $V$ is a non-empty open set, there is a $p\in V$  and $\epsilon>0$ so that $B(p,\frac{\epsilon}{2})\subset V$.
As $X$ is a physical attractor, there exists a compact subset $K$ of $W$ so that $d_H(\phi^j(K),X)\to 0$ as $j\to \infty$.
Thus, there exists $J_K\in \mathbb{N}$ so that $d_H(\phi^j(K),X)<\epsilon$, for every $j>J_K$.
This implies  $\phi^j(K)\cap V \neq \emptyset$, and so  $\phi^j(W)\cap V \neq \emptyset$, for all $j>J_K$.
Hence, $\phi$ is topologically mixing.

Conversely, assume  $\phi$ is topologically mixing.  For any 
$\epsilon>0$, let $W\subset X$ be a non-empty open set and let $\{B_i\}_{i\in \mathbb{N}}$ be an open cover 
of $X$ with $\text{diam}(B_i)<\epsilon /2$, for every $i\in \mathbb{N}$.
As $W$ is an open subset of metric space $X$, choose  a non-empty open set $W'\subset W$ such that $\overline{W'}\subset W$.
Since $\phi$ is topologically mixing, for each $i\in \mathbb{N}$, there exists $n_i\in \mathbb{N}$ so that 
$\phi^n(W')\bigcap B_i \neq\emptyset$, for all $ n>n_i.$
By compactness of $X$, there exists a finite subcover $m\in \mathbb{N}$ with $\bigcup_{j=1}^m B_{i_j}=X$.
Set
$J_K=\max\{i_j:~j=1,\dots,m\}.$
For $n > J_K$, $\phi^n(\overline{W'}) \cap B_{i_j} \neq \emptyset$ for all $j$, and so $d_H(\phi^n(\overline{W'}), X) < \epsilon$. Thus, $d_H(\phi^n(\overline{W'}), X) \to 0$, and $X$ is a physical attractor.

\textit{(ii)} 
Let $X$ be an attractor for $\phi$. For any non-empty open set $U\subset X$, there exist 
a compact set $K\subset U$. Since $X$ is an attractor, 
$d_h(\phi^{n}(K), X)\to 0$, as $n\to \infty$.
Thus, $X$ is a physical attractor.

\textit{(iii)} 
Assume that \(X\) is an attractor for \(\phi\).
Suppose there exists a family \(\{ B_n \}_{n \in \mathbb{N}}\) 
such that \(\Ls^* B_n \neq \emptyset\). 
Then there exists a compact set \(K \neq X\) 
such that \(N(B_n: U, K, V)\) is infinite for some disjoint open sets 
\(U \supset K\) and \(V\). 
Thus, there exist sequences \(\{ i_\ell \}_{\ell \in \mathbb{N}}\) 
with \(i_\ell \to \infty\) and \(\{ j_\ell \}_{\ell \in \mathbb{N}}\)  in natural numbers
such that \(U \cap \phi^{i_\ell}(B_{j_\ell}) \neq \emptyset\) and 
\(V \cap \phi^{i_\ell}(B_{j_\ell}) = \emptyset\). Since \(\mathcal{K}(X)\) is compact, 
\(\{ \phi^{i_\ell}(B_{j_\ell}) \}_\ell\) has a convergent subsequence with limit \(B \in \mathcal{K}(X)\). 
This contradicts \(d_H(\phi^n(B), X) \to 0\), as \(n\to \infty\). Hence, \(\Ls^* B_n = \emptyset\), and \(X\) is a proper attractor.

\textit{(iv)} Assume \(X\) is an attractor for \(\phi\). Suppose \(\mathrm{Eq}(\phi) \neq X\), 
so there exists a sensitive point \(x \in X\) with \(\epsilon > 0\) 
such that every neighborhood of \(x\) has \(d_\phi\)-diameter at least \(\epsilon\). 
Thus, there exist sequences \(\{ x_i \}_{i \in \mathbb{N}} \subset X\) with
 \(d(x_i, x) \to 0\) and \(\{ n_i \}_{i \in \mathbb{N}} \subset \mathbb{N}\) with
  \(n_i \to \infty\) such that \(d_H(\phi^{n_i}(x), \phi^{n_i}(x_i)) > \epsilon\). 
  Since \(\phi\) is continuous, \(\{ \phi^{n_i}(x_i) \}_{i \in \mathbb{N}}\) 
  has a convergent subsequence in \(\mathcal{K}(X)\) with limit \(Y \neq X\). 
Set \(B_i = \{ x_i \}\), for all $i \in \mathbb{N}$. Then \(Y \in \Ls^* B_i\), 
contradicting the proper attractor condition (\(\Ls^* B_i = \emptyset\)). 
Thus, \(\mathrm{Eq}(\phi) = X\), and \(\phi\) is equicontinuous.

\textit{(v)} 
Assume  $X$ is an attractor for $\phi$. As $X$ is compact, by item (iv) of this theorem,  \(\phi\) is equicontinuous.
Let $W$ be a non-empty open subset of $X$.
By compactness of $X$, the phase space  $X$ is an attractor for equicontinuous relation $\phi$.
So, for every $x\in X$ there exist $\delta_x>0$ and  
$n_x\in\mathbb{N}$ so that $\phi^n(y)\cap W\neq\emptyset$, for every $n\geq n_x$ and
$y\in B(x,\delta_x)$.

Since $X$ is compact and  $\phi$ is continuous, 
there exist $\{x_i\}_{i=1}^k\subset X$, $\{\delta_{x_i}\}_{i=1}^k\subset (0,\infty)$ 
and $\{n_{x_i}\}_{i=1}^k\subset \mathbb{N}$ so that 
$ \{B(x_i,\delta_{x_i})\}_{i=1}^k$ is an open cover of $ X$ and  $\phi^n(y)\cap  W\neq\emptyset$, for all $n\geq\max_{1\leq i \leq k} n_{x_i}$ and for every $y\in X$.

This argument shows that $\phi^{-n}(W)=X$, for every $n\geq\max_{1\leq i \leq k} n_{x_i}$
which finishes the proof of topological exactness of $\phi^{-1}$.

Conversely, assume  $\phi^{-1}$ is topologically exact. For any   $\epsilon>0$, let
$(V_i)_{i=1}^s$  be a finite open cover of $X$ with 
$\text{diam}(V_i)<\frac{\epsilon}{2}$, for all $i=1,\dots,s.$
As $\phi^{-1}$ is topologically exact,  for  each $i$,  there exists 
$n_i\in \mathbb{N}$ so that $\phi^{-n}(V_i)=X$, for every $n>n_i.$
Thus,  $\phi^n(x)\bigcap V_i\neq\emptyset$, for every $x\in X$ and $n>n_i$.
For any $x\in X$, $d_H(\phi^n(x),X)<\epsilon$, for all $n>\max\{n_i:\  1\leq i\leq s\}$.
Hence, $X$ is an attractor and so it is proper attractor.

\textit{(vi)} 
Assume  $X$ is an attractor for $\phi$.
By Item (iii), $X$ is a proper attractor.
 Suppose by way of contradiction  for some $\epsilon>0$ and every  $n\in\mathbb{N}$, 
we have $d_H(\phi^i(x),X)\geq\epsilon$, for some  $i>n$ and some $x\in X$.
For every $n\in \mathbb{N}$, choose $i_n\in \mathbb{N}$ and $x_n$ in $X$ so that 
$d_H(\phi^{i_n}(x_n),X)\geq\epsilon$. Take $B_n=\{x_n\}$, for every $n\in \mathbb{N}$.
Since $\{\phi^{i_n}(B_n)\}_n$ is a sequence of compact sets and $\mathcal{K}(X)$ is compact space, 
$\Ls^* B_n\neq \emptyset$, which is a contradiction with definition of proper attractor.
\end{proof}

\begin{proof}[Proof of Theorem~\ref{contersen}]\textit{(i)} 
Since \(X\) is compact, let \(\{ U_n \}_{n \in \mathbb{N}}\) be a countable basis of \(X\).  
Fix \(x_0\) in \(X\).
By Theorem~\ref{topexactattractor}, if \(\phi\) is topologically exact, \(X\) is an attractor for \(\phi^{-1}\).
So, there exists \(j_1 \in \mathbb{N}\) such that \(\phi^{-i}(x_0)\bigcap U_1\neq \emptyset\), for every \(i\geq j_1\). 
Fix \(x_{-j_1}\in \phi^{-j_1}(x_0)\bigcap U_1\). Clearly, \(\phi^{j_1}(x_{-j_1})=x_0\).
By repeating the mentioned process, we can take  subsequences 
\(\{j_n\}_{n\in \mathbb{N}}\) in \(\mathbb{N}\) and
\(\{x_{-t_n}\}_{n\in \mathbb{N}}\) in \(X\) so that  
 \(t_1=j_1\), \(t_n=\sum_{i=1}^nj_i\) and \(\phi^{j_n}(x_{-t_n})=x_{-t_{n-1}}\), for every \(n\geq 2\).
 Construct a sequence 
\(\mathbf{x} = (\dots,x_{-i},\dots, x_{-2}, x_{-1}, x_0) \in \mathcal{S}\),
by selecting
$$x_i=\left\{
  \begin{array}{ll}
    x_{-t_n}, & \hbox{if  $ \ \ \ \exists\ t_n, \ i=t_n$; } \\
    \phi^{i-t_n}(x_{-t_n}), & \hbox{if $ \ \ \ \exists \ t_n, \ t_n<i<t_{n-1}$,}
  \end{array}
\right.
$$
for every  $i\in\mathbb{N}$.
 Since \(\{ U_n \}\) is a basis, the set \(\{ x_{-i} : i \in \mathbb{N} \}\) is dense in \(X\).

\textit{(ii)} 
Fix $x_0$. As the Item (i), we can construct a sequence $\mathbf{x}$ in $\mathcal{S}$. 
In this way,  we used this fact that: for any \(\epsilon > 0\), 
there exists \(j \in \mathbb{N}\) such that \(\phi^j(B(x_0, \epsilon)) = X\). 
Thus, there exists \(x_{-j} \in \phi^{-j}(x_0) \cap B(y_0, \epsilon)\).

Similarly, there exists \(y_{-j} \in \phi^{-j}(y_0) \cap B(x_0, \epsilon)\). 
To complete the proof, it is enough to order $\{U_i\}$ in term of diameters so that $\liminf_i\text{diam}(U_i)=0$.
Thus the  sequences \(\{ x_{-i} \}\) and \(\{ y_{-i} \}\) can be chosen such that \(\liminf_{i \to \infty} d(x_{-i}, y_{-i}) = 0\).

\end{proof}
%



\section{Topologically mixing and topologically exact IFSs}\label{sec3}
To illustrate  significant  topological exact systems, 
we appeal to the Hutchinson operator induced by 
iterated function systems (IFSs).
Assume that $\mathcal{F}=\{f_1,\dots,f_k\}$  is an 
iterated function system~(IFS) consisting in a finite  family of continuous maps $f_i$ on the compact metric space $X$.
The  IFS $\mathcal{F}$ is called topologically exact,
 if 
$
 {\widetilde{F}}^n(U)=X,
$
for every open subset $U$ of $X$ and some $n\in\mathbb{N}$.
The IFS $\mathcal{F}$ is called topologically mixing, if  for every open subset $W$ and $V$, there exist $J\in \mathbb{N}$ so that
$\widetilde{F}^j(W)\bigcap V\neq \emptyset$, for every  $ j \geq J$.

We now generalize  notions of chain transitivity and shadowing to characterize topological mixing in IFSs.
A sequence
$\xi=\{x_i\}_{i=0}^\infty \subset X$ is a $\delta -$chain of  $\mathcal{F}$ if
there exists $\omega\in \{1,\dots,k\}^\mathbb{N}$ so that 
$d(x_{i+1}, f_{\omega_i}(x_i)) < \delta$, for all $i \in\mathbb{N}$.
A $\delta -$chain from $x$ to $y$ in $X$  is a finite $\delta -$chain 
 $\xi=\{x_i\}_{i=0}^n $ so that $x_0=x$ and $x_n=y$.
An IFS $\mathcal{F}$ is  \emph{chain transitive} if for any $\delta>0$ and any two points $x$ and $y$
there exists a $\delta -$chain from $x$ to $y$.
An IFS $\mathcal{F}$  is  $\delta-$chain mixing, if there exists $N \in \mathbb{N}$
such that  for each $x, y \in X$ and each $n > N$, there exists a $\delta -$ chain  with length $n$ from $x$ to $y$. 
Also, an IFS $\mathcal{F}$ is \emph{chain mixing} if it is $\delta - $ chain mixing, for every $\delta> 0$.
An IFS $\mathcal{F}$  has the \emph{shadowing property}
if for
each $\epsilon > 0$ there exists $\delta>0$ such that for 
any $\delta$ pseudo-orbit $\xi = \{x_i\}$ driven by $\omega$ one can find a point
$y$ in $X$ such that $d(x_i,f_\omega^i(y))<\epsilon$, for all $i\in \mathbb{N}$.

It is also worth noting that  the IFS type of Corollary 14 in   \cite{RW08}  can be found in \cite{Nia}.
\begin{theorem}[Theorem 3.7 \cite{Nia}]\label{Niath}
Let   $\mathcal{F}=\{f_1,\dots,f_k\}$ be an IFS  on  $X$. Then 
$\mathcal{F}$ is chain mixing if and only if it is chain transitive.
\end{theorem}

The classical topological notions  chain transitivity and shadowing property imply that
the phase space is  physical attractor.

\begin{proof}[Proof of Theorem~\ref{attfixed}]
Let   $\mathcal{F}=\{f_1,\dots,f_k\}$ be a chain transitive  IFS  on  $X$ with the shadowing property.
By Theorem~\ref{Niath}, the IFS $\mathcal{F}$ is chain mixing.
Fix $x\in X$.  For any $\epsilon>0$ there exists  $\delta>0$  so that every $\delta -$chain can be $\epsilon$-shadowed.
Let  $V_1,\dots,V_m\subseteq X$ be open sets covering $X$  with $\max_{1\leq i\leq m}\text{diam}(V_i)<\epsilon$.
Choose $y_i\in V_i$, for every $i=1,\dots,m $.
Since, $\mathcal{F}$ is chain mixing, there exists an integer  $N_{i,\epsilon}$ so that for 
each $n > N_{i,\epsilon}$, there exists a $\delta - $ chain of length $n$ from $x$ to $y_i$. 
Put $N_\epsilon=\max\{N_{i,\epsilon}:~ 1\leq i\leq m\}$.

By the shadowing property, for every $1\leq i \leq m$, every $\delta$-chain of length  $n>N_\epsilon$ from $x$ to $y_i$ 
$\epsilon$-shadowed by some $z(i)\in B(x,\epsilon)$.
Thus, for $n>N_\epsilon$,
\[
d_H(F^n(\overline{B(x,\epsilon)}), X)<\epsilon.
\]
For any set $B$ with non-empty interior, choose $x$  in the interior of $B$ and  $0<r$ so that $B(x,r)\subset B$.
Moreover, by choosing $\epsilon<r$  and repeating the mentioned process  for $x$,
$F^n(\overline{B(x,\epsilon)})\subset F^n(B)$ and so 
\[
d_H(F^n(\mathcal{B}), X)<\epsilon; \ \ \ \forall~n>N_\epsilon.
\]
Since $\epsilon$ is arbitrary, we have $X$ is a physical attractor.

\end{proof}
A non-empty subset $A$ of $X$ is said to be backward (resp. forward) invariant
for an  IFS $\mathcal{F}=\{f_1,\dots,f_k\}$ if for all $f \in \mathcal{F}^+$ the following holds:
\[\emptyset\neq f^{-1}(A)=\{x\in X: f(x)\in A\} \subseteq A\ (\text{resp.}\ f(A)\subseteq A).\] 
An IFS $ \mathcal{F}$ is backward minimal (resp. forward minimal), i.e. the unique backward (resp. forward)
invariant non-empty closed set is the whole space $X$.

Now, let $f:X\to X $ be a map on $X$ and $id_X:X\to X$ be the identity map on $X$.
Notice that the iteration of the extension of Hutchinson operator of $\{f,id_X\}$ record all  previous actions.
\begin{proposition}[recording IFSs]\label{record}
Let  $\mathcal{F}$ be a minimal IFS.
Then  the  IFS $\mathcal{F}\cup id_X$ is topologically exact 
and so $X$ is a proper attractor for it.
\end{proposition}
\begin{proof}
Assume that  $U$ is a non-empty open subset of $X$.
Since IFS $\mathcal{F}$ is minimal,
one can find $h_1,\dots,h_r$ in $\mathcal{F}^+$
such that $\{h_i(U) : 1 \leq i \leq r\}$ 
is an open cover of $X$. 
Take $\ell=\max\{|h_i|: 1 \leq i \leq r\}$, where $|h|$ 
denote the
minimal natural number $n$ with $h\in \mathcal{F}^n$.
Stand  $\widetilde{F}$ for the extension of Hutchinson IFS $\mathcal{F}$ and 
$\widetilde{F}_*$ for the extension of Hutchinson IFS $\mathcal{F}\cup id_X$.

Note that, $\widetilde{F}_*= id_X\cup \widetilde{F}$ and 
$\widetilde{F}_*^j(U)=\bigcup_{i=0}^j \widetilde{F}^i(U)$, for all $j\in \mathbb{N}$.
Thus, for every $i\geq \ell$, one can have 
$\widetilde{F}_*^i(U)=X$, which is complete the proof.
\end{proof}
Let an IFS $ \mathcal{F}$ generated by a family of homeomorphisms $\{f_1,\dots,f_k\}$
on a compact metric space $X$.
An IFS $ \mathcal{F}$ is \emph{totally forward minimal} if 
IFS $ \mathcal{F}^n$ is forward minimal, for every $n\in\mathbb{N}$.
Take $\mathcal{F}_-:=\{f_1^{-1},\dots,f_k^{-1}\}$, 
where $f_i^{-1}$ is the inverse of $f_i$, for every $i=1,\dots,k$.
Stand  $\widetilde{F}_-$ for the extension of Hutchinson operator of $\mathcal{F}_-$. 
We say that the IFS $\mathcal{F}$ is \emph{symmetric} if for each $f\in\mathcal{F}$ it holds that $f^{-1}\in \mathcal{F}$.
\begin{theorem}
Every  symmetric totally forward minimal IFS $\mathcal{F}=\{f_1,\dots,f_k\}$ of a family of homeomorphisms
is topologically exact and so, $X$ is a proper attractor for $\mathcal{F}_-$.
 \end{theorem}
 \begin{proof}
Since IFS $\mathcal{F}$ is symmetric, $id_X\in \mathcal{F}^2$.
By assumption  $\mathcal{F}^2$ is minimal and so
by Theorem~\ref{record} the IFS $\mathcal{F}^2\cup id_X=\IFS\mathcal{F}^2$ is topologically exact,
i.e. for every open set $U$, $\widetilde{F}^{2n}(U)=X$, for some $n\in \mathbb{N}$.
This implies that IFS $\mathcal{F}$ is  topologically exact.
 \end{proof}

 but totally forward and totally backward minimality of an IFS  does not guarantee 
 topologically exactness of it.
\begin{example}
While total forward minimality ensures topological exactness for symmetric iterated function systems (IFSs), this property  does not hold in general. 
However, total backward and total forward minimality of an IFS $\mathcal{F}$ does not guarantee its topological exactness.
For instance, consider two circle diffeomorphisms $f_1$ and $f_2$ that are rotations by angles $\alpha, \beta \in \mathbb{R} \setminus \mathbb{Q}$, respectively, 
with $\beta - \alpha \in \mathbb{Q}$. Let $\mathcal{F} = \{f_1, f_2\}$ and $\mathcal{F}_- = \{f_1^{-1}, f_2^{-1}\}$. 
For every $n \in \mathbb{N}$, the map $f_{n\alpha}$ is minimal and belongs to $\mathcal{F}^n$, 
so the IFS $\mathcal{F}$ is totally minimal. Similarly, the IFS $\mathcal{F}_-$ is totally minimal.
However, by Corollary 4.10 in \cite{BC23}, the circle is not an attractor for either IFS $\mathcal{F}$ or IFS $\mathcal{F}_-$.
Thus, by Theorem~\ref{topexactattractor}, neither IFS $\mathcal{F}$ nor IFS $\mathcal{F}_-$ is topologically exact.
\end{example}
Assume \(\mathcal{F} = \{f_1, \dots, f_k\}\) is a backward minimal IFS with a repelling fixed point, where each \(f_i: X \to X\) is a homeomorphism on a compact metric space \(X\). Consequently, \(\mathcal{F}_- = \{f_1^{-1}, \dots, f_k^{-1}\}\) is a forward minimal IFS with the Hutchinson operator \(\widetilde{F}_-(A) = \bigcup_{i=1}^k f_i^{-1}(A)\). It is important to note that Main Corollary~\ref{replling} cannot be directly derived from the results in \cite{Sari}, as the Hutchinson operators \(\widetilde{F}(A) = \bigcup_{i=1}^k f_i(A)\) and \(\widetilde{F}_-\) are not inverses. Specifically, the compositions \(\widetilde{F}_- \circ \widetilde{F}\) and \(\widetilde{F} \circ\widetilde{F}_-  \)
include the identity and other non-identity compositions. Consequently, Main Theorem~\ref{replling} cannot be directly derived from  results in \cite{Sari}, as the forward and backward dynamics exhibit distinct behaviors.
\begin{proof}[Proof of Theorem~\ref{replling}]
Let \(f \in \mathcal{F}\) have a repelling fixed point \(p\).
For any open set \(W \subset X\), 
since \(\mathcal{F}\) is backward minimal, there exists 
\(g \in \mathcal{F}^+\) such that \(p \in g(W) \subseteq \widetilde{F}^{|g|}(W)\). 
As $p$ is repelling fixed point for $f$, choose an open set \(U \ni p\) such that \(U \subset f(U) \cap g(W)\). 
Clearly,  \(U \subset \widetilde{F}^i(U)\) for all \(i \in \mathbb{N}\). 
Since \(\mathcal{F}\) is backward minimal, 
there exist \(h_1, \dots, h_s \in \mathcal{F}^+\) 
such that \(\bigcup_{i=1}^s h_i(U) = X\). 
Let \(m = \max\{ |h_i| : 1 \leq i \leq s \}\), where \(|h|\) 
is the minimal \(n \in \mathbb{N}\) such that \(h \in \mathcal{F}^n\). 
Thus, \(\widetilde{F}^m(U) = X\). 
Hence, \(\widetilde{F}^{m+|g|}(W) = X\), so \(\mathcal{F}\) is topologically exact.
\end{proof}
\begin{figure}
\begin{center}
\label{fig:101}
\begin{tikzpicture}[domain=-1:12.5, scale=0.5,very thick]
\draw (-1,0) -- (5,0);
\draw[dashed](5,0)--(7,0);
\draw[->](7,0)--(12.5,0) node[right] {$x$};
\draw[->] (0,-1) -- (0,12.5) node[above] {$y$};
\draw (12,0) -- (12,12);  
\draw[color=black] plot (12,0) node[below] {$1$};
\draw (0,12) -- (6,12);
\draw[dashed](6,12)--(8,12);
\draw(8,12)--(12,12);
\draw[color=black] plot (0,12) node[left] {$1$};

\draw[color=black] plot (3,0) node[below] {$x_1$};
\draw[dashed] (3,0) -- (3,12);

\draw[color=black] plot (9,0) node[below] {$x_2$};
\draw[dashed] (9,0) -- (9,12);

\draw[color=black] plot (0,9) node[left] {$y_1$};
\draw[dashed] (0,9) -- (3,9);

\draw[color=black] plot (0,11) node[left] {$y_2$};
\draw[dashed] (0,11) -- (3,11);

\draw[color=black] plot (0,2) node[left] {$f_1^m(x_1,y_1)$};
\draw[dashed] (0,2) -- (3,2);

\draw[color=black] plot (0,4) node[left] {$f_1^m(x_1,y_2)$};
\draw[dashed] (0,4) -- (3,4);

\draw[color=black] plot (12,6) node[right] {$f_1^m(x_2,y_1)$};
\draw[dashed] (9,6) -- (12,6);

\draw[color=black] plot (12,8) node[right] {$f_1^m(x_2,y_2)$};
\draw[dashed] (9,8) -- (12,8);

\draw[color=red][thick] (3,4) -- (3.72,12);
\draw[color=red] (3,3.4) -- (3.79,12);
\draw[color=red] (3,2.7) -- (3.82,12);
\draw[color=red] (3,2) -- (3.9,12);

\draw[color=red][thick](3.7,0)-- (4.8,12);
\draw[color=red][thick](3.77,0)-- (4.86,12);
\draw[color=red][thick](3.84,0)-- (4.92,12);
\draw[color=red][thick](3.9,0)--(4.98,12);

\draw[color=red](4.8,0)-- (5.88,12);
\draw[color=red](4.9,0)-- (5.96,12);
\draw[color=red](4.98,0)--(6.06,12);

\draw[color=black] plot (5.6,6) node[right] {$\bullet\bullet\bullet$};
\draw[color=red](9,6)-- (8.45,0);
\draw[color=red](9,7)-- (8.36,0);
\draw[color=red](9,8)--(8.27,0);

\draw[color=red](7.37,0)-- (8.45,12);
\draw[color=red](7.28,0)-- (8.36,12);
\draw[color=red](7.19,0)--(8.27,12);

\end{tikzpicture}

\end{center}
\caption{For some $m\in\mathbb{N}$, the action of $f^m$ on the square bounded by 
$x=x_1$, $x=x_2$, $y=y_1$ and $y=y_2$ determined by the red region. }
\label{asgard}
\end{figure}
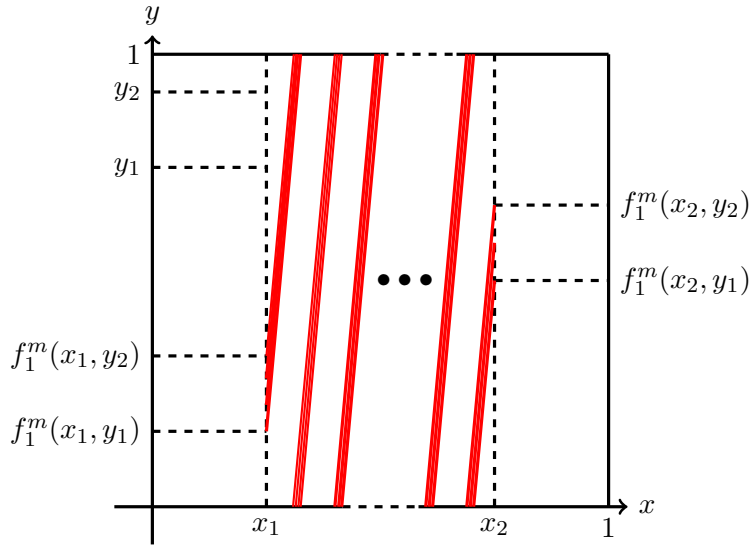

\begin{example}\label{ex:mixing}
Consider the 2-torus $\mathbb{T}^2 = \mathbb{R}^2 / \mathbb{Z}^2$ with the metric induced from $\mathbb{R}^2$. Define homeomorphisms $f_1, f_2: \mathbb{T}^2 \to \mathbb{T}^2$ by:
\[
f_1(x, y) = (x, y + x), \quad f_2(x, y) = (x + y, y),
\]
and let $\mathcal{F} = \{f_1, f_2\}$. The extension of the Hutchinson operator is $\widetilde{F}(A) = f_1(A) \cup f_2(A)$ for $A \subset \mathbb{T}^2$. Consider the point $p = (1/2, 1/2)$. Compute its iterates:
\begin{align*}
  f_1(1/2, 1/2) &= (1/2, 1/2 + 1/2) = (1/2, 0), \\
  f_2(1/2, 1/2) &= (1/2 + 1/2, 1/2) = (0, 1/2), \\
  f_1(1/2, 0) &= (1/2, 0 + 1/2) = (1/2, 1/2), \\
  f_2(1/2, 0) &= (1/2 + 0, 0) = (1/2, 0), \\
  f_1(0, 1/2) &= (0, 1/2 + 0) = (0, 1/2), \\
  f_2(0, 1/2) &= (0 + 1/2, 1/2) = (1/2, 1/2).
\end{align*}
Thus, the orbit of $p$ under $\mathcal{F}$ is the finite set $\{(1/2, 1/2), (1/2, 0), (0, 1/2)\}$. Since $\widetilde{F}^n(\{p\})$ is contained in this finite set for all $n$, it does not cover $\mathbb{T}^2$, so $\mathcal{F}$ is not topologically exact. By Theorem \ref{topexactattractor}(iv), $\mathbb{T}^2$ is not an attractor for $\mathcal{F}_- = \{f_1^{-1}, f_2^{-1}\}$.

Consider open sets $U, V \subset \mathbb{T}^2$. Notice that points with irrational coordinates have dense orbits. Thus, for some $x \in U$, there exists an $n$ and a sequence $\omega = (\omega_1, \ldots, \omega_n) \in \{1, 2\}^n$ such that:
\[
f_\omega^n(x) = f_{\omega_n} \circ \cdots \circ f_{\omega_1}(x) \in V.
\]
Hence, $\mathcal{F}$ is transitive.
To show $\mathcal{F}$ is topologically mixing, see Figure \ref{asgard} for an illustration of $f_1^m$ acting on a square region, stretching it vertically. Similarly, consecutive iterations of $f_2$ stretch a rectangular region horizontally. Thus, for any open sets $U, V \subset \mathbb{T}^2$, there exists $N \in \mathbb{N}$ such that for all $n \geq N$, $\widetilde{F}^n(U) \cap V \neq \emptyset$, confirming definition of topological mixing.
\end{example}

\textbf{Acknowledgments.} We thank Abbas Fakhari, Pablo G. Barrientos and Joel A. Cisneros
for useful discussions and suggestions.


\begin{thebibliography}{99}
\bibitem{A}
E. Akin,   The general topology of dynamical systems, Graduate studies in mathematics, vol 1, AMS 1993.
\bibitem{Blr}
 M. F. Barnsley, K. Lesniak and  M. Rypka, Chaos game for IFSs on topological spaces, Journal of Mathematical
Analysis and Applications,  435(2), 1458–1466, 2016.
\bibitem{Blw}
 M. F. Barnsley, K. Lesniak and  C. Wilson, Some Recent Progress Concerning Topology of Fractals, Recent
Progress in General Topology III, 69–92, 2014.
\bibitem{Bv11}
 M. F. Barnsley and A. Vince, The chaos game on a general iterated function system, Ergodic Theory and Dynamical Systems,
31, 1073–1079, 2011.
\bibitem {Bv13}
 M. F. Barnsley and A. Vince, Developments in fractal geometry, Bulletin of the Australian Mathematical Society, 3, 299–348, 2013.
\bibitem{BC23}
P. G. Barrientos and J. A. Cisneros,
Minimal strong foliations in skew-products of iterated function systems, arxiv:2304.11229
\bibitem{bgms}
P. G. Barrientos, F. H. Ghane, D. Malicet and A. Sarizadeh, On the chaos game of iterated function
systems,   Topological Methods in Nonlinear Analysis, 49(1), 105-132, 2017.
\bibitem{M}
Milnor J., On the Concept of Attractor, Communications in Mathematical  Physics, 99, 177-195, 1985.
\bibitem{Nia}
M. Fatehi Nia,  Chain Mixing and Chain Transitive Iterated Function Systems,
 Journal of Dynamical Systems and Geometric Theories, 18(2), 211-221, 2020.
\bibitem{RW08}
D. Richeson and J. Wiseman,  Chain recurrence rates and topological entropy,
Topology and its Applications 156, 251–261, 2008.
\bibitem {RT}
D. Ruelle, F. Takens, On the nature of turbulence, Communications in Mathematical  Physics, 20, 343-344, 1971.
\bibitem{R1}
D. Ruelle, Strange attractors, The Mathematical Intelligencer, 2, 126-137, 1979-80.
\bibitem{R2}
D. Ruelle, Small random perturbations of dynamical systems and the definition of attractor,
Communications in Mathematical  Physics, 82, 137-151, 1981.
\bibitem{Sari}
A. Sarizadeh, Attractor for minimal iterated function systems, Collectanea Mathematica, 76, 105-111,2025.
\end{thebibliography}
\end{document}